\documentclass[12pt,a4paper]{amsart}

\usepackage[colorlinks, linkcolor=red,anchorcolor=blue,citecolor=green]{hyperref}
\usepackage{graphics,epic}
\usepackage{amsmath,amssymb, amsthm}
\usepackage[all,2cell]{xy}
\usepackage{times}
\usepackage{caption}
\usepackage{tikz}
\usetikzlibrary{positioning,arrows.meta}
\usetikzlibrary{matrix,calc}
\usetikzlibrary{decorations.pathmorphing}
\usepackage{bbold}

\newtheorem{theorem}{Theorem}[section]
\newtheorem*{theorem*}{Theorem}

\newtheorem{lemma}[theorem]{Lemma}
\newtheorem{proposition}[theorem]{Proposition}

\newtheorem*{conjecture*}{Conjecture}

\theoremstyle{definition}

\newtheorem{example}[theorem]{Example}
\newtheorem{remark}[theorem]{Remark}
\newtheorem{definition}[theorem]{Definition}

\newcommand{\opname}[1]{\operatorname{\mathsf{#1}}}

\renewcommand{\P}{\mathbb{P}}

\newcommand{\id}{\mathbf{1}}

\newcommand{\Aut}{\opname{Aut}}

\begin{document}
	
\title{Are cluster automorphism groups finitely generated?}	
\author{Changjian Fu}
\address{Changjian Fu\\Department of Mathematics\\Sichuan University\\610064 Chengdu\\P.R.China}
\email{changjianfu@scu.edu.cn}
\author{Zhanhong Liang}
\address{Zhanhong Liang\\Department of Mathematics\\Sichuan University\\610064 Chengdu\\P.R.China}
\email{zhanhongliang@stu.scu.edu.cn}
\author{Yinzhi Wang}
\address{Yinzhi Wang\\Department of Mathematics\\Sichuan University\\610064 Chengdu\\P.R.China}
\email{yinzhiwang@stu.scu.edu.cn}
\thanks{}

\begin{abstract}
This paper investigates the finite generation of cluster automorphism groups. By applying the pseudo $\mathbb{N}$-grading introduced in our previous work, we establish a sufficient condition for a cluster automorphism group to be finitely generated. As applications, we re-establish the finite generation of the automorphism groups for all finite mutation type cluster algebras, and verify the acyclic cases. Furthermore, we illustrate through examples that our approach significantly simplifies the computation of presentations for these groups in certain cases.

\end{abstract}
\maketitle

\section{Introduction}
Cluster algebras were introduced by Fomin and Zelevinsky \cite{FZ02} to provide a combinatorial framework for studying canonical bases in quantum groups and total positivity in algebraic groups. A cluster algebra is a commutative algebra generated by a set of cluster variables, which are organized into overlapping sets of equal finite cardinality known as clusters. Starting from an initial cluster, the entire collection of clusters is generated through a recursive process called mutation.

The cluster automorphism group was introduced by Assem, Schiffler, and Shramchenko \cite{ASS12} to capture the symmetries of a cluster algebra. This group consists of automorphisms that permute the clusters while preserving the mutation structure. This group has proven to be useful in investigating the structural theory of cluster algebras \cite{BM16,BMP20,CL2020}. The cluster automorphism groups have been explicitly determined for several prominent classes of cluster algebras. Assem et al. \cite{ASS12} investigated these groups for acyclic skew-symmetric cluster algebras by leveraging the representation theory of the underlying path algebras, providing explicit computations for the simply-laced Dynkin and Euclidean types. Subsequently, Chang and Zhu \cite{CZ16a} determined these groups for cluster algebras of finite type using a folding technique.

Parallel to these developments, Blanc and Dolgachev \cite{BD15} computed the full automorphism groups for cluster algebras of rank two, which contain the cluster automorphism groups as subgroups. Furthermore, the automorphism groups of cluster algebras arising from marked surfaces have been fully characterized through a series of foundational works \cite{ASS12, Gu11, BS15, DL23}. Such cluster algebras are known to be of finite mutation type \cite{FST12a}. Beyond those originating from surfaces, Fraser \cite{Fr20} obtained presentations for the cluster automorphism groups of two additional cluster algebras of finite mutation type, namely $E_7^{(1,1)}$ and $E_8^{(1,1)}$. Ishibashi \cite{I20} introduced and studied saturated cluster modular groups, which feature cluster automorphism groups as their quotients. He obtained finite presentations for these groups in the cases of types $X_6$ and $X_7$.
The first two authors \cite{FL25} extended the methodology of \cite{ASS12} to develop an elementary framework for computing these groups in specific settings; in particular, they provided a complete classification for cluster algebras of rank 3 with indecomposable exchange matrices (cf. \cite{FT26}). Despite these advances, a universal method for computing the cluster automorphism group of an arbitrary cluster algebra remains elusive.

In this paper, we investigate cluster automorphism groups from a group-theoretic perspective. It has been observed that all previously computed cluster automorphism groups are finitely generated. Our objective is to provide a theoretical explanation for this phenomenon and to present new classes of cluster algebras whose cluster automorphism groups are likewise finitely generated. Specifically, we establish a sufficient condition that guarantees the finite generation of the cluster automorphism group (Lemma \ref{lem:fin-gen}). It is worth noting that all previously known examples satisfy this condition or a variation thereof. As a corollary, we deduce that the cluster automorphism group of any finite mutation type or acyclic cluster algebra is finitely generated (Theorem \ref{thm:main-mutation-finite} and Theorem \ref{thm:main-ayclic}). We remark that while the finite generation of cluster automorphism groups for cluster algebras of finite mutation type was previously established by Fock and Goncharov \cite{FG06} (cf. also \cite[Lemma 2.15]{I20} and \cite[Lemma 10.1]{Fr20} ), our approach is conceptually distinct and offers a novel perspective on identifying generators. In the final section, we provide examples to illustrate how our proof can be applied to calculate cluster automorphism groups for certain cluster algebras.

\subsection*{Acknowledgements}
The authors are grateful to Dr. Chris Fraser for pointing out the work of Fock and Goncharov \cite{FG06} and for providing the reference \cite{Fr20}. They also thank Dr. Tsukasa  Ishibashi for his interest and valuable comments.
This work is partially supported by the National Natural Science Foundation of China (Grant No. 12571040).

\section{Preliminary}
\subsection{Cluster algebra}
In this section, we recall basics of cluster algebras with trivial coefficients. We follow \cite{FZ02,FZ07}. Throughout this section, we fix a positive integer $n$ and let $\mathcal{F}$ be the field of rational function in $n$ variables with coefficients in $\mathbb{Q}$. For a real number $b$, we denote by $[b]_+=\max(b,0)$.

 An integer matrix $B\in {\rm M}_n(\mathbb{Z})$ is {\em skew-symmetrizable} if there exists a diagonal matrix $D=\operatorname{diag}(d_1,...,d_n)$ whose diagonal entries $d_i$ are positive integers such that $DB$ is {\em skew-symmetric}, i.e., $(DB)^T=-DB$.  For a skew-symmetrizable matrix $B=(b_{ij})\in {\rm M}_{n}(\mathbb{Z})$, we associate a weighted directed graph/quiver $\Gamma_B$ to $B$, whose vertex set is $\{1,\ldots,n\}$ and there is an arrow $i\to j$ whenever $b_{ij}>0$, and this arrow is assigned the weight $\sqrt{-b_{ij}b_{ji}}$.
 We say that $B$ is {\em indecomposable} if the directed graph $\Gamma_B$ is connected, and $B$ is {\em acyclic} if the directed graph has no oriented cycles.

%  \begin{definition}
%      Let $B=(b_{ij})$ be an $n\times n$ skew-symmetrizable matrix and $k\in \{1,\ldots,n\}$. The mutation $\mu_k$ in direction $k$ transforms $B$ into a new matrix $\mu_k(B):=(b_{ij}')$ defined as follows:
%      \begin{align}
% b_{ij}^\prime &=\begin{cases}-b_{ij} & i=k\;\text{or}\;j=k;\\
%  b_{ij}+b_{ik}[b_{kj}]_++[-b_{ik}]_+b_{kj}&\text{otherwise.}\end{cases}
% \end{align}
% \end{definition}

\begin{definition}[Seed]
   A {\em labeled seed} (of rank $n$) in $\mathcal{F}$ is a pair $\Sigma=(\mathbf{x},B)$ such that $\mathbf{x}=(x_1,...,x_n)$ is a transcendence basis of $\mathcal{F}$ over $\mathbb{Q}$ and $B$ is an $n\times n$ skew-symmetrizable matrix. We refer to $\mathbf{x}$ and $B$ as the {\em cluster} and {\em exchange matrix} of $\Sigma$, respectively.
\end{definition}
\begin{definition}[Seed mutation]
    For any $k \in \{1,\dots,n\}$, the {\em mutation $\mu_k$ in direction $k$} of the labeled seed $\Sigma=(\mathbf{x},B)$ is a new labeled seed $\Sigma'=(\mathbf{x'},B')$ defined as follows:
\begin{align*}
x'_i&=
    \begin{cases}
    x^{-1}_k\left( \prod\limits^n_{j=1} x^{[b_{jk}]_+}_j + \prod\limits^n_{j=1}x^{[-b_{jk}]_+}_j  \right) &\text{$i=k$};\\
    x_i &\text{else}. 
    \end{cases}\\
    b'_{ij}&=
    \begin{cases}
        -b_{ij} &\text{$i=k$ or $j=k$};\\
        b_{ij}+b_{ik}[b_{kj}]_++[-b_{ik}]_+b_{kj} &\text{else}.
    \end{cases}
\end{align*}
We write $\mu_k(\Sigma):=\Sigma'$ and the matrix $\mu_k(B):=B'$ is also called the {\em mutation of $B$ in direction $k$}.
\end{definition}
It is straightforward to check that $\mu_k$ is an involution, i.e., $\mu_k^2(\Sigma)=\Sigma$.

% We say that two seeds are mutation-equivalent if one can be obtained from the other by finite sequence of mutations. Similarly, one defines the mutation equivalence for clusters and exchange matrices.
% We denote by $\operatorname{Mut}(B)$ the mutation equivalence classes of $B$, i.e., the matrices which can be obtained from $B$ by finite sequence of mutations.
\begin{remark}
    The correspondence $B\mapsto \Gamma_B$ is not one-to-one, but if the skew-symmetrizer $D$ is given, then the matrix $B$ is uniquely determined by $\Gamma_B$. We may transform the mutation of exchange matrices to the mutation of weighted directed graphs, and vice versa; see \cite{FZ03}.
\end{remark}

Let $\mathbb{T}_n$ be the $n$-{\em regular tree} such that each vertex has $n$ edges attached to it, and the edges are labeled by $1,\dots,n$ in such a way that the $n$ edges incident to each vertex have distinct labels. We write $\xymatrix{t\ar@{-}[r]^k&t'}$ to indicate an edge labeled by $k$.

\begin{definition}[Cluster pattern]
A {\em cluster pattern} of rank $n$ in $\mathcal{F}$ is a collection $\mathbf{\Sigma }=\{\Sigma_t=(\mathbf{x}_t,B_t)\}_{t\in\mathbb{T}_n}$ of labeled seeds in $\mathcal{F}$ such that for any edge $\xymatrix{t\ar@{-}[r]^k&t'}$,  $\Sigma_{t^\prime}=\mu_k(\Sigma_t)$. 
% We also refer to $\{B_t\}_{t\in \mathbb{T}_n}$ a $B$-pattern.
\end{definition}
A cluster pattern $\mathbf{\Sigma}$ is uniquely determined by assigning $\Sigma_{t_0}=(\mathbf{x}_{t_0},B_{t_0})$  to the vertex $t_0$. We refer to $t_0$ as the {\em root vertex}, $\mathbf{x}_{t_0}$ as the {\em initial cluster}, and $B_{t_0}$ as the {\em initial exchange matrix} of this cluster pattern.
For a given cluster pattern $\mathbf{ \Sigma }=\{\Sigma_t=(\mathbf{x}_t,B_t)\}_{t\in\mathbb{T}_n}$, we always write
\[
\mathbf{x}_t=(x_{1;t},\dots, x_{n;t}), B_t=(b_{ij;t}), 
\] and for the root vertex $t_0$, we further denote
\[\mathbf{x}:=\mathbf{x}_{t_0}:=(x_1,\dots,x_n), B:=B_{t_0}:=(b_{ij}).
\]
We call $x_{i;t}$ a {\em cluster variable} of $\mathbf{\Sigma}$, and denote by $\mathcal{X}(\mathbf{\Sigma})$ the set of all cluster variables of $\mathbf{\Sigma}$. 
\begin{definition}[Cluster algebra]
    Let  $\mathbf{ \Sigma }$ be a cluster pattern of rank $n$ in $\mathcal{F}$. The {\em cluster algebra} $\mathcal{A}:=\mathcal{A}(\mathbf{\Sigma})$ associated with $\mathbf{\Sigma}$ is the $\mathbb{Z}$-subalgebra of $\mathcal{F}$ generated by $\mathcal{X}(\mathbf{\Sigma})$. We refer to $n$ as the rank of $\mathcal{A}$. We say that $\mathcal{A}$ is {\em indecomposable} if $B_t$ is indecomposable for some $t\in \mathbb{T}_n$, and  $\mathcal{A}$ or $\mathbf{\Sigma}$ is {\em acyclic} if $B_t$ is acyclic for some $t\in \mathbb{T}_n$.
\end{definition}

We conclude this subsection by recalling the definition of the $S_n$-action on labeled seeds, which plays a crucial role in the definition of cluster automorphisms.
\begin{definition}[$S_n$-action]
Let $\Sigma=(\mathbf{x},B)$ be a labeled seed of rank $n$ and $\sigma \in S_n$ be a permutation.
 We define the action of $\sigma$ on $\Sigma$ by
\[
    \sigma \Sigma=(\sigma \mathbf{x},\sigma B),
\]   % 标点符号位置
where $\sigma \mathbf{x}:=(x'_1,\dots,x'_n)$, $\sigma B:=(b'_{ij})^n_{i,j=1}$ are defined by
\[
    x'_i=x_{\sigma^{-1}(i)} \quad \text{and}\quad b'_{ij}=b_{\sigma^{-1}(i)\sigma^{-1}(j)}.
\] %!!

\end{definition}
The $S_n$-action induces an equivalence relation on labeled seeds. For a labeled seed $\Sigma$, we denote $[\Sigma]$ as its equivalence class and refer to it as an {\em unlabeled seed}. For the exchange matrices $B$ and $B'$, we write $B \sim B'$ if there exists $\sigma \in S_n$ such that $B' = \sigma(B)$. The $S_n$-action is compatible with seed mutations.

\begin{lemma}\label{lemma:Sn-action-formula}
Let $\Sigma=(\mathbf{x},B)$ be a labeled seed of rank $n$, $1\leq k \leq n$ and $\sigma\in S_n$. Then $\sigma(\mu_k(\Sigma))=\mu_{\sigma(k)}(\sigma(\Sigma))$.
\end{lemma}

\subsection{Cluster automorphism}\label{ss:cluster-autom}
In this section, we recall the basics of cluster automorphisms. Throughout, we fix a cluster pattern $\mathbf{\Sigma}$ of rank $n$ with root vertex $t_0$, and let $\mathcal{A} = \mathcal{A}(\mathbf{\Sigma})$ denote the corresponding cluster algebra. 

\begin{definition}[Cluster automorphism]
An automorphism of $\mathbb{Z}$-algebras $f:\mathcal{A}\rightarrow \mathcal{A}$ is called a {\em cluster automorphism} if there exist two seeds $\Sigma=(\mathbf{x},B)$, $\Sigma'=(\mathbf{x}',B') \in \mathbf{\Sigma}$ and a permutation $\sigma \in S_n$ such that the following conditions are satisfied:
\begin{itemize}
    \item[(CA1)] $\sigma(f(\mathbf{x}))=\mathbf{x}'$;
    \item[(CA2)] $\sigma(f(\mu_k(\mathbf{x})))=\mu_{\sigma(k)}(\mathbf{x}')$ for any $k \in \{1,\dots,n\}$.
\end{itemize}
\end{definition}

% For any two seeds $\Sigma=(\mathbf{x},B)$, $\Sigma'=(\mathbf{x}',B')$ and $\sigma \in S_n$, because the elements of each cluster form a transcendence basis  of $\mathcal{F}$, there exists a unique automorphism $f$ of $\mathbb{Z}$-algebras induced by $x_i \longmapsto x'_{\sigma(i)}$. Therefore, (CA1) can be reformulated by $f(\sigma(\mathbf{x}))=\mathbf{x}'$. Furthermore, $f$ is not a cluster automorphism in general since condition (CA2) is not satisfied. In the following proposition, it will point out the equivalent conditions of (CA2).
The following equivalent characterization of cluster automorphisms is useful; see \cite[Lemma 2.3]{ASS12} or \cite[Lemma 3.3]{FL25}.
\begin{proposition}\label{prop:equivalent-def-cluster-auto}
Let $f:\mathcal{A}\rightarrow\mathcal{A}$ be a $\mathbb{Z}$-algebra automorphism. The following conditions are equivalent:
\begin{itemize}
    \item[(1)] $f$ is a cluster automorphism.
    \item[(2)] There exist two seeds $\Sigma=(\mathbf{x},B)$, $\Sigma'=(\mathbf{x}',B')$ and a permutation $\sigma \in S_n$ such that $\sigma(B)=B'$ or $\sigma(B)=-B'$. In this case, $f$ is uniquely determined by $f(x_i)=x_{\sigma(i)}'$ for $1\leq i\leq n$.
    % \[
    % \sigma(f(\mathbf{x}))=\mathbf{x}'; \quad \sigma(B)=B' \;\text{or}\; \sigma(B)=-B'.
    % \]
    % \item[(3)] Let $D=diag(d_1,\dots,d_n)$ be a skew-symmetrizer of $B$. Then there exists two seeds $\Sigma=(\mathbf{x},B)$, $\Sigma'=(\mathbf{x}',B')$ and a permutation $\sigma \in S_n$ such that:
    % \[  
    % \sigma(f(\mathbf{x}))=\mathbf{x}'; \quad \sigma(\Gamma(B))=\Gamma(B') \;\text{or}\; \sigma(\Gamma(B))=(\Gamma(B'))^{op}
    % \]
    % and it holds that $d_i=d_{\sigma(i)}$ for any $i \in \{1,\dots,n\}$.
\end{itemize}
\end{proposition}

% \begin{remark}
% There is one more characterization of cluster automorphisms: a $\mathbb{Z}$-algebra automorphism $f:\mathcal{A}\rightarrow\mathcal{A}$ is a cluster automorphism if and only if $f$ maps each cluster to a cluster. Here, we do not discuss this in detail. %参考文献
% \end{remark}

As a direct consequence of Proposition \ref{prop:equivalent-def-cluster-auto},  each cluster automorphism $f$ can be represented by a quadruple 
\[f:=(\Sigma_1,\Sigma_2,\sigma,\epsilon),\] where $\Sigma_1=(\mathbf{x}_1,B_1)$ and $\Sigma_2=(\mathbf{x}_2,B_2)$ are labeled seeds, $\sigma\in S_n$ and $\epsilon\in \{\pm\}$, such that $\sigma(B_1)=\epsilon B_2$. In the following, we denote by $[B]$ the equivalence class of $B$ consisting of exchange matrices $\pm\sigma(B)$, $\forall \sigma\in S_n$.

% All in all, a quadruple $(\Sigma_1,\Sigma_2,\sigma,\varepsilon)$ consisting of two seeds: $\Sigma_1=(\mathbf{x}_1,B_1)$ and $\Sigma_2=(\mathbf{x}_2,B_2)$ of $\mathcal{A}$, $\varepsilon \in \{\pm \}$ and a permutation $\sigma$ such that $\sigma(B_1)=\varepsilon B_2$ determines a unique cluster automorphism $f$ which satisfies $\sigma(f(\mathbf{x}_1))=\mathbf{x}_2$. Hence we represent a cluster automorphism of $\mathcal{A}$ by a quadruple in the form above. As we shall see, a cluster automorphism can be written as different quadruples. In fact, we have the following results.

For $s,t\in \mathbb{T}_n$, we denote by $p(s,t)$ the unique path from $s$ to $t$ in $\mathbb{T}_n$:
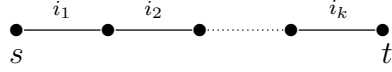
\[
\xymatrix@R=0.1cm@C=1.0cm@M=0cm
{
\bullet \ar@{-}[r]^{i_1}&\bullet \ar@{-}[r]^{i_2} &\bullet \ar@{.}[r]&\bullet\ar@{-}[r]^{i_k}&\bullet.\\
s &&&  &t
}
\]
In this case, we denote
\begin{itemize}
    \item $\ell(s,t)=k$;
    \item $\mu_{p(s,t)}:=\mu_{i_k}\cdots\mu_{i_2}\mu_{i_1}$;
    \item $t:=\mu_{p(s,t)}(s)$;
    \item $\mu_{\sigma(p(s,t))}:=\mu_{\sigma(i_k)}\cdots\mu_{\sigma(i_2)}\mu_{\sigma(i_1)}$ for any $\sigma\in S_n$.
\end{itemize}
% Before that, we introduce some notations. From now on, fix $\mathbf{\Sigma}=\{\Sigma_t=(\mathbf{x}_t,B_t)\}_{t\in \mathbb{T}_n}$ to be a cluster pattern of rank $n$ with the root vertex $t_0$. Recall that for arbitrary $s,t \in \mathbb{T}_n$, there exists a path $p(s,t)=i_1 \dots i_k$ from $s$ to $t$, i.e.,
% \[
% \xymatrix@R=0.1cm@C=1.0cm@M=0cm
% {
% \bullet \ar@{-}[r]^{i_1}&\bullet \ar@{-}[r]^{i_2} &\bullet \ar@{.}[r]&\bullet\ar@{-}[r]^{i_k}&\bullet.\\
% s &&&  &t
% }
% \]
% We denote $t:=\mu_{i_k}\dots\mu_{i_2}\mu_{i_1}(s)=\mu_{p(s,t)}(s)$ and $\mu_{\sigma(p(s,t))}=\mu_{\sigma(i_k)} \dots \mu_{\sigma(i_1)}$. For a fixed vertex $t' \in \mathbb{T}_n$, the {\em weight of $p(s,t)$ with respect to $t'$} is defined by
% \[
%     \mathbf{w}_{t'}(s,t):=\# \{\; t''\in \mathbb{T}_n\;|\;\text{ $t''$ lies in the path $p(s,t)$ and $[B_{t''}]=[B_{t'}]$}\; \}.
% \]
The following is an observation of \cite[Proposition 3.4]{FL25}.
\begin{proposition}\label{prop:expression-composition-cluster-auto}
Let $f=(\Sigma_{t_1},\Sigma_{t_2},\sigma,\varepsilon)$ be a cluster automorphism of $\mathcal{A}$.
\begin{itemize}
    \item[(1)] For any $t_3 \in \mathbb{T}_n$, the following equality holds
    \[
    (\Sigma_{t_1},\Sigma_{t_2},\sigma,\varepsilon)=
    (\mu_{p(t_1,t_3)}(\Sigma_{t_1}),\mu_{\sigma(p(t_1,t_3))}(\Sigma_{t_2}),\sigma,\varepsilon).
    \]
    \item[(2)] Let $g=(\Sigma_{t_2},\Sigma_{t_4},\tau,\delta)$ be a cluster automorphism of $\mathcal{A}$, then  the composition of $f$ with $g$ is given by:
    \[
    (\Sigma_{t_2},\Sigma_{t_4},\tau,\delta)\circ (\Sigma_{t_1},\Sigma_{t_2},\sigma,\varepsilon)=(\Sigma_{t_1},\Sigma_{t_4},\tau\sigma,\delta\varepsilon).
    \]
\end{itemize}
\end{proposition}

Clearly, the identity on $\mathcal{A}$ is a cluster automorphism. By Proposition \ref{prop:expression-composition-cluster-auto}, $(\Sigma_{t_2},\Sigma_{t_1},\sigma^{-1},\varepsilon)$ is the inverse of $(\Sigma_{t_1},\Sigma_{t_2},\sigma,\epsilon)$. If follows that the set $\Aut(\mathcal{A})$ of all cluster automorphisms of $\mathcal{A}$ is a group under composition, which is called the {\em cluster automorphism group} of $\mathcal{A}$, see \cite{ASS12}.
Another  consequence of Proposition \ref{prop:expression-composition-cluster-auto} (1) is that, every cluster automorphism $f$ can be expressed as
\[
    (\Sigma_{t_0}, \Sigma_t,\sigma, \varepsilon)
\]
for some labeled seed $\Sigma_t=(\mathbf{x}_t,B_t)$ such that $\sigma (B_{t_0})=\varepsilon B_t$. 

\subsection{Pseudo $\mathbb{N}$-grading}Keep the notation as in Section \ref{ss:cluster-autom}.
Let $s,t\in \mathbb{T}_n$.
 For a fixed vertex $t' \in \mathbb{T}_n$, the {\em weight of $p(s,t)$ with respect to $t'$} is defined by
\[
    \mathbf{w}_{t'}(s,t):=\# \{\; t''\in \mathbb{T}_n\;|\;\text{ $t''$ lies in the path $p(s,t)$ and $[B_{t''}]=[B_{t'}]$}\; \}.
\]
For any non-negative integers $m$, we define
\[
\mathcal{P}_m(t_0)=\{\;p(t_0,t)\;|\;\mathbf{w}_{t_0}(t_0,t)=m+1\;\text{and}\;[B_t]=[B_{t_0}]\}.
\]
Note that $\mathcal{P}_0(t_0)=\{p(t_0,t_0)\}$. Moreover, the collection of all cluster automorphisms determined by $p(t_0,t) \in \mathcal{P}_m(t_0)$ is denoted by
\[
    G_m(t_0):=\left\{f=(\Sigma_{t_0},\Sigma_s,\sigma, \varepsilon)\in \Aut(\mathcal{A})~\middle|~\begin{array}{l} 
    s\in \mathbb{T}_n\;\text{s.t.}\;p(t_0,s)\in \mathcal{P}_m(t_0)
    \end{array}
    \right\}\subset \Aut(\mathcal{A}).
\]
Clearly,  $G_0(t_0)$ is a finite subgroup of $\Aut(\mathcal{A})$ and $\Aut(\mathcal{A})=\bigcup_{m\in \mathbb{N}}G_m(t_0)$, which is referred to as a {\em pseudo $\mathbb{N}$-grading} of $\Aut(\mathcal{A})$, cf. \cite{FL25}.

\begin{lemma}\cite[Lemma 3.6]{FL25}\label{lem:factorization-cluster-auto}
Let $s,t \in \mathbb{T}_n$ be vertices such that $[B_{t_0}]=[B_s]$ and $[B_{t_0}]=[B_t]$. Let $f=(\Sigma_{t_0},\Sigma_t,\sigma,\varepsilon)$, $g=(\Sigma_{t_0},\Sigma_s,\tau,\delta) \in \Aut(\mathcal{A})$. Then there exists a cluster automorphism denoted by
\[
h=(\Sigma_{t_0},\mu_{\tau^{-1}(p(s,t))}(\Sigma_{t_0}), \tau^{-1}\sigma, \delta\varepsilon)
\]
such that $f=g\circ h$. Furthermore, $\mathbf{w}_{t'}(t_0,\mu_{\tau^{-1}(p(s,t))}(t_0))=\mathbf{w}_{t'}(s,t)$ for any $t'\in \mathbb{T}_n$.
\end{lemma}

Since each path $p(t_0,t)$ satisfying $[B_t]=[B_0]$ can determine at least one cluster automorphism. Then we arbitrarily fix a cluster automorphism denoted by $f_{p(t_0,t)}$ for the path $p(t_0,t)$. Hence, others corresponding to $p(t_0,t)$  are the composition of $f_{p(t_0,t)}$ with an automorphism belonging to $G_0(t_0)$ according to Lemma \ref{lem:factorization-cluster-auto}. We denote
\[
H_m(t_0):=\{\;f_{p(t_0,t)}\;|\;p(t_0,t)\in \mathcal{P}_m(t_0)\}\subset G_m(t_0).
\]
The following is one of the main results of \cite{FL25}, which leads to an elementary approach to calculate cluster automorphism groups for  cluster algebras of lower rank.
\begin{theorem}\cite[Theorem 3.10]{FL25}\label{thm:generator-set}
    The cluster automorphism group $\operatorname{Aut}(\mathcal{A})$ is generated by $G_0(t_0)\cup H_1(t_0)$.
\end{theorem}

\section{Main results}\label{s:main-results}
\subsection{Cluster algebras of finite mutation type}

In this section, we demonstrate that the cluster automorphism group of a cluster algebra of finite mutation type is finitely generated. In fact, we establish a sufficient condition to guarantee the finite generation of this group in a broader context.

 Let $\mathbf{\Sigma}$ be a cluster pattern of rank $n$ with root vertex $t_0\in \mathbb{T}_n$.
We set
\[
    \mathcal{B}(t_0)=\{[B_s]\mid s\;\text{lies in path} \;p(t_0,t)\in \mathcal{P}_1(t_0)\}.
\]
Then we have the following result.
\begin{lemma}\label{lem:fin-gen}
    Let $\mathcal{A}$ be the cluster algebra corresponding to $\mathbf{\Sigma}$. If  $\mathcal{B}(t_0)$ is a finite set, then $\Aut(\mathcal{A})$ is finitely generated.
\end{lemma}
\begin{proof}
   By the condition, we may assume that $\mathcal{B}(t_0)=\{[B_{t_0}],[B_1],\dots,[B_m]\}$ for some nonnegative integer $m$. For $1\leq i\leq m$, we define
    \[
    l_i=\min \{ \ell(t_0,t_i)|[B_{t_i}]=[B_i],t_i\in \mathbb{T}_n \}
    \]
and denote by 
        \[l=\operatorname{max}\{l_1,\dots,l_m\}.\]
        By Theorem \ref{thm:generator-set}, it suffices to show that $H_1(t_0)$ is generated by a finite set. Indeed, we can divide $H_1(t_0)$ into two subsets as follows
   \begin{align*}
       H_1(t_0)^{< 2l+2 }:=\{f_{p(t_0,t)}\mid p(t_0,t)\in \mathcal{P}_1(t_0), \ell(t_0,t)<2l+2 \},&    \\
       H_1(t_0)^{\ge 2l+2 }:=\{f_{p(t_0,t)}\mid p(t_0,t)\in \mathcal{P}_1(t_0), \ell(t_0,t)\ge 2l+2 \}.&
   \end{align*}
    The finiteness of the first subset is evident. We claim that the second subset is generated by $G_0(t_0)\cup H_1(t_0)^{<2l+2}$, and the result follows directly.    

    We prove the claim by induction on the length $\ell(t_0,v)$ of $p(t_0,v)\in \mathcal{P}_1(t_0)$ with $\ell(t_0,v)\geq 2l+2$. Let $f_{p(t_0,t)}\in H_1(t_0)^{\ge 2l+2}$ and suppose that the claim is verified for length less than $\ell(t_0,t)$.
    We choose a vertex $t'\in p(t_0,t)$ such that $\ell(t_0,t')\ge l +1$ and  $\ell(t',t)\ge l +1$. By assumption, there is a $1\leq j\leq m$ such that $[B_{t'}]=[B_{j}]$. According to the defintion of $l_i$, there 
     exists a vertex $s\in \mathbb{T}_n$ such that $\ell(t',s)=l_j$ and $[B_{s}]=[B_{t_0}]$. Clearly, we have
    \begin{align*}
     \ell(t_0,s)\leq & \ell(t_0,t')+\ell(t',s)< \ell(t_0,t')+\ell(t',t)=\ell(t_0,t), \\
     \ell(s,t)\leq & \ell(s,t')+\ell(t',t)<\ell(t_0,t')+\ell(t',t)=\ell(t_0,t).
     \end{align*}
     Moreover, $p(t_0,s)\in \mathcal{P}_1(t_0)$ and $\mathbf{w}_{t_0}(s,t)=2$.
     Let $g=f_{p(t_0,s)}$. By Lemma \ref{lem:factorization-cluster-auto}, there exists an $h=(\Sigma_{t_0},\Sigma_{t''}, \tau,\sigma)\in G_1(t_0)$ such that $f=g\circ h$, where $p(t_0,t'')\in \mathcal{P}_1(t_0)$ with $\ell(t_0,t'')=\ell(s,t)<\ell(t_0,t)$. Furthermore, $h$ can be written as a product of $f_{p(t_0,t'')}$ with an element in $G_0(t_0)$. By induction, $f_{p(t_0,t)}$ can be generated by $G_0(t_0)\cup H_1(t_0)^{<2l+2}$. This completes the proof.

\end{proof}
 Recall that $\mathcal{A}$ is a cluster algebra of {\em finite mutation type} if the set $\{B_t\mid t\in \mathbb{T}_n\}$ is finite.
As a direct consequence of Lemma \ref{lem:fin-gen}, we recover the following well known result for cluster algebras of finite mutation type \cite{FG06}, cf. also \cite[Lemma 10.1]{Fr20}.
\begin{theorem}\label{thm:main-mutation-finite}
    Let $\mathcal{A}$ be a cluster algebra of finite mutation type and $\Aut(\mathcal{A})$ be the cluster automorphism group. Then $\Aut(\mathcal{A})$ is finitely generated.
\end{theorem}

\begin{remark}\label{rem:fin-mut-class}
\begin{itemize}
    \item If the set $\mathcal{B}(t_0)$ is finite, then $\mathcal{B}(t)$ is finite for any  $t\in \mathbb{T}_n$.
    \item It is clear that if $\mathbf{\Sigma}$  is a cluster pattern of rank $\leq 2$, then the set $\mathcal{B}(t)$ is finite. On the other hand, according to \cite{FL25}, for any cluster pattern $\mathbf{\Sigma}$ of rank $3$, the set $\mathcal{B}(t)$ is also finite. However, there are cluster patterns of higher ranks  such that $\mathcal{B}(t)$ is infinite; see Example \ref{exam:infinite-exchange-matrix}.
\end{itemize}

\begin{remark}
 Cluster algebras of finite mutation type were classified by Fomin, Shapiro, and Thurston \cite{FST12a, FST12b}. In particular, any skew-symmetric cluster algebra of finite mutation type with rank $n \ge 3$ is either of surface type or one of the eleven exceptional types.The cluster automorphism groups of cluster algebras of surface type were determined in \cite{DL23}, based on the mapping class groups of the associated surfaces. Furthermore, eight of the eleven exceptional types (specifically, types $E$, $\tilde{E}$, $E_7^{(1,1)}$ and $E_8^{(1,1)}$) have been explicitly characterized \cite{ASS12,Fr20}. Consequently, a natural question arises: can we explicitly determine the cluster automorphism groups for the remaining three cases? The proof of Lemma \ref{lem:fin-gen} provides, at the very least, a valid algorithm for identifying their sets of generators, cf. Example \ref{exam:type-X7}.
\end{remark}
    
\end{remark}
%  This leads us to pose the following question. 
% \begin{question}
% For any cluster pattern $\mathbf{\Sigma}$ of rank $n$,
%   does there exist a vertex $t_0\in \mathbb{T}_n$ such that $\mathcal{B}(t_0)$ is finite?
% \end{question}

\subsection{Acyclic cluster algebras}
Throughout this subsection, let $\mathbf{\Sigma}$ be an acyclic cluster pattern of rank $n$, and denote by $t_0\in \mathbb{T}_n$ such that the exchange matrix $B_{t_0}$ is acyclic. Denote by $\mathcal{A}$ the cluster algebra associated with $\mathbf{\Sigma}$. As we will see in Example \ref{exam:infinite-exchange-matrix} that the finiteness of $\mathcal{B}(t_0)$ is not always satisfied for $\mathbf{\Sigma}$. Nevertheless, the finiteness of a specific subset of $\mathcal{B}(t_0)$ is sufficient to establish that $\Aut(\mathcal{A})$ is finitely generated.

For  a seed $\Sigma_t$ and $k\in \{1,\ldots, n\}$, if $b_{jk;t}\geq 0$ for all $j\in \{1,\ldots, n\}$, then $\mu_k$ is called a {\em sink mutation} of $\Sigma_t$. Dually, if $b_{jk;t}\leq 0$ for all $j\in \{1,\ldots, n\}$, then $\mu_k$ is called a {\em source mutation} of $\Sigma_t$.
For a path
\[
\xymatrix{t_1\ar@{-}[r]^{i_1}&t_2\ar@{-}[r]^{i_2}&t_3\ar@{-}[r]&\cdots\ar@{-}[r]^{i_k}&t_{k+1}}
\]
of $\mathbb{T}_n$, we say that the path $p(t_1,t_{k+1})$ is a {\em sink-source sequence} if $\mu_{i_j}$ is a sink or source mutation of $\Sigma_{t_j}$ for $1\leq j\leq k$.

The following is a direct consequence of Lemma \ref{lemma:Sn-action-formula}.
\begin{lemma}\label{lem:refine-factorization}
    Let $s,t \in \mathbb{T}_n$ be vertices such that $[B_{t_0}]=[B_s]$ and $[B_{t_0}]=[B_t]$. Let $f=(\Sigma_{t_0},\Sigma_t,\sigma,\varepsilon)$, $g=(\Sigma_{t_0},\Sigma_s,\tau,\delta) \in \Aut(\mathcal{A})$.
    If both $p(t_0,t)$ and $p(t_0,s)$ are sink-source sequences, then $p(t_0,
    \mu_{\tau^{-1}(p(s,t))}(t_0)$ is also a sink-source sequence.
\end{lemma}

The following result has been proved by \cite[Corollary 4]{CK06} for skew-symmetric cases, but the proof is valid for skew-symmetrizable cases; see also \cite[Lemma 3.5]{ASS12}.

\begin{lemma}\label{lem:APR-mut}
If $t\in \mathbb{T}_n$ such that $B_t$ is also acyclic, then there is a vertex $s\in \mathbb{T}_n$ such that $[\Sigma_s]=[\Sigma_t]$ and the path $p(t_0,s)$ is a sink-source sequence.
\end{lemma}

For any positive integer $m$, we define
\[
\mathcal{P}_m^{ss}(t_0)=\{p(t_0,t)\in \mathcal{P}_m(t_0)\mid \text{$p(t_0,t)$ is labeled by a sink-source sequence}\}
\]
and
\[
H_m^{ss}(t_0)=\{f_{p(t_0,t)}\mid p(t_0,t)\in \mathcal{P}_m^{ss}(t_0)\}\subseteq H_m(t_0).
\]
The following is a refinement of Theorem \ref{thm:generator-set} for acyclic cluster algebras.
\begin{lemma}\label{l:gen-set-acyclic}
    Let $\mathcal{A}$ be the associated cluster algebra of $\mathbf{\Sigma}$ and $\Aut(\mathcal{A})$ be the cluster automorphism group. Then $\Aut(\mathcal{A})$ is generated by $G_0(t_0)\cup H_1^{ss}(t_0)$.
\end{lemma}
\begin{proof}
    According to Theorem \ref{thm:generator-set}, it suffices to prove that any $f_{p(t_0,t)}\in H_1(t_0)$ can be expressed as a product of elements of $G_0(t_0)\cup H_1^{ss}(t_0)$.

    Let $f=f_{p(t_0,t)}\in H_1(t_0)$. It follows that $B_t$ is acyclic by Proposition \ref{prop:equivalent-def-cluster-auto}. Hence, there is a vertex $s\in \mathbb{T}_n$ such that $[\Sigma_s]=[\Sigma_t]$ and $p(t_0,s)$ is labeled by a sink-source sequence by Lemma \ref{lem:APR-mut}.  Again by Proposition \ref{prop:equivalent-def-cluster-auto}, $f$ can be expressed as $f=(\Sigma_{t_0},\Sigma_s, \sigma,\epsilon)$ for some $\sigma\in S_n$ and $\epsilon\in \{\pm\}$. We conclude that $f=f_{p(t_0,s)}\circ h$ for some $h\in G_0(t_0)$ by Lemma \ref{lem:factorization-cluster-auto}, where $f_{p(t_0,s)}\in H_m^{ss}(t_0)$ for some $m\geq 1$. By induction on $m$,  we conclude that $f_{p(t_0,s)}$ is generated by $G_0(t_0)\cup H_1^{ss}(t_0)$ by applying Lemma \ref{lem:factorization-cluster-auto} and \ref{lem:refine-factorization}.
\end{proof}
\begin{remark}
    Lemma \ref{l:gen-set-acyclic} can be used to reduce the lengthy computation of cluster automorphism groups for acyclic cluster algebras of rank $3$ in \cite{FL25}, cf. Example \ref{eg:rank-3-acyclic}.
\end{remark}
Set \[\mathcal{B}^{ss}(t_0)=\{[B_s]\mid \text{$s\in p(t_0,t)$ and $ p(t_0,t)\in \mathcal{P}_1^{ss}(t_0)$}\}.\]
\begin{lemma}
    The set $\mathcal{B}^{ss}(t_0)$ is  a finite set.
\end{lemma}
\begin{proof}
    Note that if $p(t_0,t)\in \mathcal{P}_1^{ss}(t_0)$, then $B_s$ is acyclic for any $s\in p(t_0,t)$. It follows that the set $\mathcal{B}^{ss}(t_0)$ is finite.
\end{proof}
Now we are in the position to state the main result of this subsection, which is essentially implicit in \cite{ASS12}.
\begin{theorem}\label{thm:main-ayclic}
The cluster automorphism group $\Aut(\mathcal{A})$ is finitely generated.
\end{theorem}
\begin{proof}
 
The proof is similar to the one of Lemma \ref{lem:fin-gen}. Namely, let \[\mathcal{B}^{ss}(t_0)=\{[B_{t_0}],[B_1],\ldots, [B_m]\},\]  and for each $1\leq i\leq m$, define
\[
l_i=\min \{\ell(t_0,t_i)\mid [B_{t_i}]=[B_i], \text{$p(t_0,t_i)$ is labeled by a sink-source sequence}\}.
\]
Set $l=\max\{l_1,\ldots, l_m\}$. By Lemma \ref{l:gen-set-acyclic}, it suffices to show that $H_1^{ss}(t_0)$ is generated by a finite set. We divide $H_1^{ss}(t_0)$ into the following subsets:
 \begin{align*}
       H_1^{ss}(t_0)^{< 2l+2 }:=\{f_{p(t_0,t)}\mid p(t_0,t)\in \mathcal{P}_1^{ss}(t_0), \ell(t_0,t)<2l+2 \},&    \\
       H_1^{ss}(t_0)^{\ge 2l+2 }:=\{f_{p(t_0,t)}\mid p(t_0,t)\in \mathcal{P}_1^{ss}(t_0), \ell(t_0,t)\ge 2l+2 \}.&
   \end{align*}
   It suffices to prove that the second subset is generated by $G_0(t_0)\cup H_1^{ss}(t_0)^{<2l+2}$.
   Assume that we have shown that $f_{p(t_0,u)}$ is generated by $G_0(t_0)\cup H_1^{ss}(t_0)^{<2l+2}$ for any $p(t_0,u)\in \mathcal{P}_1^{ss}(t_0)$ with $2l+2\leq \ell(t_0,u)<k$.

   Let $p(t_0,t)\in \mathcal{P}_1^{ss}(t_0)$ with $\ell(t_0,t)=k$. We choose a vertex $t'\in p(t_0,t)$ such that $\ell(t_0,t')\geq l+1$ and $\ell(t',t)\geq l+1$. By the definition of $\mathcal{B}^{ss}(t_0)$, there is a $1\leq j\leq m$ such that $[B_{t'}]=[B_j]$. It follows that there exists a vertex $s\in \mathbb{T}_n$ such that $\ell(t',s)=l_j$, $[B_s]=[B_{t_0}]$ and $p(t',s)$ is labeled by a sink-source sequence. As a consequence, $p(t_0,s)\in \mathcal{P}_1^{ss}(t_0)$ and $p(s,t)$ is also labeled by a sink-source sequence. Moreover, $\ell(t_0,s)<\ell(t_0,t)$ and $\ell(s,t)<\ell(t_0,t)$. By Lemma \ref{lem:factorization-cluster-auto}, there is  an $h=(\Sigma_{t_0},\Sigma_{t''},\sigma,\epsilon)\in G_1(t_0)$ such that $f_{p(t_0,t)}=f_{p(t_0,s)}\circ h$. Furthermore, $p(t_0,t'')\in \mathcal{P}_1^{ss}(t_0)$ with $\ell(t_0,t'')=\ell(s,t)<\ell(t_0,t)$. Hence, $h$ is a product of $f_{p(t_0,t'')}$ with an element in $G_0(t_0)$. This finishes the proof.

\end{proof}

\section{Examples}
In this section, we present some examples on how to calculate cluster automorphism groups based on the proofs in Section \ref{s:main-results}. Let $f=(\Sigma_{t_0},\Sigma_t,\sigma, \epsilon)$ be a cluster automorphism. If $t=\mu_{i_k}\cdots\mu_{i_1}(t_0)$, then we denote $f$ by
\[
f=g_{i_k\cdots i_1}^{\sigma,\epsilon}.
\]
If, in addition, $t=t_0$, we simply write $f=\psi_{\sigma}^{\epsilon}$.

\begin{example}\label{eg:rank-3-acyclic}
    Let $2\leq b\leq c$.
    Let $B$ be a skew-symmetrizable matrix whose associated weighted directed graph $\Gamma_B$ is given by:
    
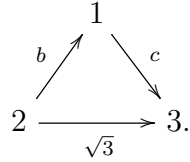
\[
    \xymatrix@C=0.5cm{&1\ar[dr]^c&\\ 2\ar[ur]^b\ar[rr]_{\sqrt{3}}&&3.}
    \]
    We remark that the matrix $B$ corresponding to $\Gamma_B$ is not unique; however, this is independent on the computation of the cluster automorphism group. Fix a cluster pattern $\mathbf{\Sigma}$ such that $B_{t_0}=B$, and let $\mathcal{A}$ be the associated cluster algebra. A direct computation shows that $|\mathcal{P}_1^{ss}(t_0)|=2$, see Figure \ref{fig:sink-source-sequence-sqrt{3}bc}. Here we present the exchange matrices by weighted directed graphs. Let $t_1=\mu_3\mu_1\mu_2(t_0)$ and $t_2=\mu_{2}\mu_1\mu_3(t_0)$.
   The set $H_1^{ss}(t_0)$ can be chosen as
    \[
    H_1^{ss}(t_0)=\{g_{312}^{\id,+},g_{213}^{\id,+}\}.
    \]
    It is easy to see that $G_0(t_0)=\{\mathbf{id}\}$ and hence $\Aut(\mathcal{A})$ is generated by $H_1^{ss}(t_0)$ by Lemma \ref{l:gen-set-acyclic}.
  Furthermore, 
  \begin{eqnarray*}
      g_{312}^{\id,+}\circ g_{213}^{\id,+}&=& (\mu_{3}\mu_1\mu_2(\Sigma_{t_0}),\mu_{3}\mu_1\mu_2(\Sigma_{t_2}),\id,+)\circ(\Sigma_{t_0}, \Sigma_{t_1},\id,+)\\
      &=&(\Sigma_{t_1},\Sigma_{t_0},\id, +)\circ (\Sigma_{t_0}, \Sigma_{t_1},\id,+)\\
      &=&\mathbf{id}.
  \end{eqnarray*}
  Hence, $\Aut(\mathcal{A})$ is generated by $f$ and we conclude that $\Aut(\mathcal{A})\cong \mathbb{Z}$ by \cite[Lemma 6.2]{FL25}.

\end{example}
  
    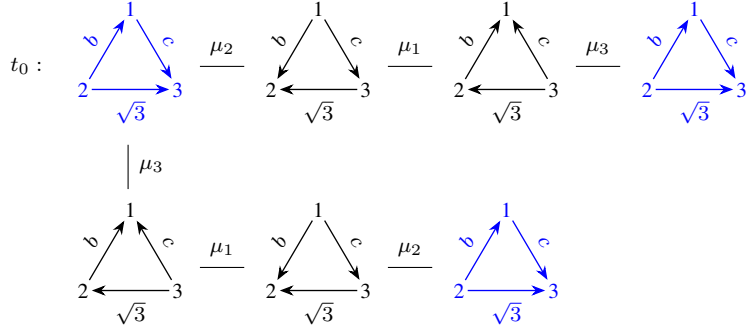
\begin{figure}
    \centering
\begin{tikzpicture}[
    >=Stealth,
    quivernode/.style={inner sep=1pt},
    quiveredge/.style={->, line width=0.5pt},
    blueq/.style={blue},
    redq/.style={red},
    orangeq/.style={orange!80!black},
    blackq/.style={black}
]

% --- Basic quiver types ---
\newcommand{\drawquivera}[1]{
  \begin{tikzpicture}[scale=1.25]
    \node[quivernode,#1,font=\tiny] (#1-1) at (0,0){1};
    \node[quivernode,#1,font=\tiny] (#1-2) at (-0.5,-0.85){2};
    \node[quivernode,#1,font=\tiny] (#1-3) at (0.5,-0.85){3};
    \draw[quiveredge,#1,font=\tiny] (#1-2)--node[midway,above,sloped, font=\tiny]{$b$}(#1-1);
    \draw[quiveredge,#1] (#1-2)--node[midway,below,font=\tiny]{$\sqrt{3}$}(#1-3);
    \draw[quiveredge,#1] (#1-1)--node[midway,above,sloped,font=\tiny]{$c$}(#1-3);
  \end{tikzpicture}
}

\newcommand{\drawquiverb}[1]{
  \begin{tikzpicture}[scale=1.25]
    \node[quivernode,#1,font=\tiny] (#1-1) at (0,0){1};
    \node[quivernode,#1,font=\tiny] (#1-2) at (-0.5,-0.85){2};
    \node[quivernode,#1,font=\tiny] (#1-3) at (0.5,-0.85){3};
    \draw[quiveredge,#1] (#1-1)--node[midway,above,sloped,font=\tiny]{$b$}(#1-2);
    \draw[quiveredge,#1] (#1-1)--node[midway,above,sloped,font=\tiny]{$c$}(#1-3);
    \draw[quiveredge,#1] (#1-3)--node[midway,below,font=\tiny]{$\sqrt{3}$}(#1-2);
  \end{tikzpicture}
}

\newcommand{\drawquiverc}[1]{
  \begin{tikzpicture}[scale=1.25]
    \node[quivernode,#1,font=\tiny] (#1-1) at (0,0){1};
    \node[quivernode,#1, font=\tiny] (#1-2) at (-0.5,-0.85){ 2};
    \node[quivernode,#1,font=\tiny] (#1-3) at (0.5,-0.85){3};
    \draw[quiveredge,#1] (#1-2)--node[midway,above,sloped,font=\tiny]{$b$}(#1-1);
    \draw[quiveredge,#1] (#1-3)--node[midway,above,sloped,font=\tiny]{$c$}(#1-1);
    \draw[quiveredge,#1] (#1-3)--node[midway,below,font=\tiny]{$\sqrt{3}$}(#1-2);
  \end{tikzpicture}
}

% --- Matrix layout ---
\matrix[matrix of nodes, nodes={anchor=center}, row sep=0.6cm, column sep=0.6cm] (m) {
\node(a1){\drawquivera{blueq}};&\node(a2){\drawquiverb{blackq}};&\node(a3){\drawquiverc{blackq}};&\node(a4){\drawquivera{blueq}};\\
\node(b1){\drawquiverc{blackq}};&\node(b2){\drawquiverb{blackq}};&\node(b3){\drawquivera{blueq}};&\\
};

% --- Connectors ---

\draw[-] (a1) -- (a2) node[midway,above] {\tiny $\mu_2$};
\draw[-] (a2) -- (a3) node[midway,above] {\tiny$\mu_1$};
\draw[-] (a3) -- (a4) node[midway,above] {\tiny$\mu_3$};

\draw[-] (a1) -- (b1) node[midway,right] {\tiny $\mu_3$};
\draw[-] (b1) -- (b2) node[midway,above] {\tiny $\mu_1$};
\draw[-] (b2) -- (b3) node[midway,above] {\tiny $\mu_2$};

\node[left=1pt of a1.west] {\tiny $t_0:$};

\end{tikzpicture}
   \caption{Sink-source sequences}
   \label{fig:sink-source-sequence-sqrt{3}bc}
\end{figure}

\begin{example}\label{exam-acyclic-cluster-algebra}
Let $a_1,a_2,b_1,b_2,c_1,c_2$ be positive integers and $B=\begin{pmatrix}
        0&a_1&0&0\\ -a_2&0&b_1&0\\ 0&-b_2&0&c_1\\ 0&0&-c_2&0
    \end{pmatrix}$. Assume that $B$ is skew-symmetrizable and $2\leq a:=\sqrt{a_1a_2}<b:=\sqrt{b_1b_2}<c:=\sqrt{c_1c_2}$.
    The associated weighted directed graph/quiver is given by
    \[
    \xymatrix{1\ar[r]^a&2\ar[r]^b&3\ar[r]^c&4.}
    \]
    Fix a cluster pattern $\mathbf{\Sigma}$ such that $B_{t_0}=B$, and let $\mathcal{A}$ be the associated cluster algebra.
    We are going to compute the cluster automorphism group of $\mathcal{A}$.

    Clearly, $\mathcal{B}^{ss}(t_0)$ consists of $4$ elements, say $[B],[B_1],[B_2],[B_3]$, where the associated weighted directed graphs are listed as follows:
    \[
    \xymatrix{\Gamma_{B_1}:1\ar[r]^{a}&2\ar[r]^b&3&4\ar[l]_c, }
    \]
    \[
    \xymatrix{\Gamma_{B_2}:1\ar[r]^a&2&3\ar[l]_b\ar[r]^c&4,}
    \]
    \[
    \xymatrix{\Gamma_{B_3}:1\ar[r]^a&2&3\ar[l]_b&4\ar[l]_c.}
    \]
    A direct computation shows that $l_1=1,l_2=2,l_3=1$ and hence $l=2$. The condition $2\leq a<b<c$ implies that $G_0(t_0)=\{\id\}$. It follows that $\Aut(\mathcal{A})$ is generated by $H^{ss}_1(t_0)^{<6}$ by the proof of Theorem \ref{thm:main-ayclic}. Denote by $\mathcal{P}_1^{ss}(t_0)^{<6}:=\{p(t_0,t)\mid p(t_0,t)\in \mathcal{P}_1^{ss}(t_0),\ell(t_0,t)<6\}$. It is routine to see that $|\mathcal{P}_1^{ss}(t_0)^{<6}|=12$ and the set $H_1^{ss}(t_0)^{<6}$ can be chosen as 
    \[
    H_1^{ss}(t_0)^{<6}=\{g_{4121}^{\id,-},g_{4321}^{\id, +},g_{1421}^{\id,-},g_{4141}^{\id,+},g_{1241}^{\id,-},g_{4341}^{\id,-},g_{1414}^{\id,+},g_{1214}^{\id,-},g_{4314}^{\id,-},g_{4134}^{\id,-},g_{1434}^{\id,-},g_{1234}^{\id,+}\}.
    \]
    Note that $\mu_1\mu_4=\mu_4\mu_1$. It follows that 
    \[g_{1414}^{\id,+}=g_{4141}^{\id,+}=\mathbf{id}, g_{4121}^{\id,-}=g_{1421}^{\id,-},\] \[g_{1241}^{\id,-}=g_{1214}^{\id,-},g_{4341}^{\id,-}=g_{4314}^{\id,-},g_{4134}^{\id,-}=g_{1434}^{\id,-}.\] On the other hand, by Proposition \ref{prop:expression-composition-cluster-auto}, we compute
    \[g_{4121}^{\id,-}\circ g_{1214}^{\id,-}=\mathbf{id},g_{4321}^{\id,+}\circ g_{1234}^{\id,+}=\mathbf{id},
g_{4314}^{\id,-}\circ g_{4134}^{\id,-}=\mathbf{id}.
    \]
    Moreover, $g_{4314}^{\id,-}\circ g_{1234}^{\id,+}=g_{1214}^{\id,-}$. Putting all of these together, we conclude that $\Aut(\mathcal{A})$ is generated by $g_{1234}^{\id,+}$ and $g_{4314}^{\id,-}$. Again by noticing that $\mu_3\mu_1=\mu_1\mu_3$, we obtain $g_{4314}^{\id,-}\circ g_{4314}^{\id,-}=\mathbf{id}$. Similar to \cite[Lemma 6.2]{FL25} or by using representations of valued quivers, one can show that the order of $g_{1234}^{\id,+}$ is $\infty$. Furthermore,
    \[
    g_{1234}^{\id,+}\circ g_{4314}^{\id,-}=g_{4314}^{\id,-}\circ g_{4321}^{\id,+}
    \]
    by noticing that $\mu_2\mu_4=\mu_4\mu_2$. Now it is routine to show that $\Aut(\mathcal{A})$ is isomorphic to 
   the infinite dihedral group $D_\infty:=\langle x,y~|~y^2=\mathbf{id}, xy=yx^{-1}\rangle$.
\end{example}

\begin{example}\label{exam:infinite-exchange-matrix}
    Let $B=\begin{pmatrix}
        0 &2 &0&0\\ -2&0&2&0\\ 0&-2 &0&2\\ 0&0&-2&0
    \end{pmatrix}$ and $\Gamma_B$ be the associated weighted directed quiver. Fix a cluster pattern $\mathbf{\Sigma}$ such that $B_{t_0}=B$ and let $\mathcal{A}$ the associated cluster algebra. It is an acyclic cluster algebra, and hence $\Aut(\mathcal{A})$ is finitely generated by Theorem \ref{thm:main-ayclic}.  

    In the case where $a=b=c=2$, which is slightly different from the previous example. A direct computation shows that the group $\Aut(\mathcal{A})$ is generated by $\psi^-_{(14)(23)}$, $g_{1234}^{\id,+}$ and $g_{4314}^{\id,-}$. Moreover, the subgroup generated by $g_{1234}^{\id,+}$ and $g_{4314}^{\id,-}$ is isomorphic to $D_\infty$, and
    \begin{align*}
    &\psi^-_{(14)(23)}\circ g_{1234}^{\id,+}\circ \psi^-_{(14)(23)}=g_{4321}^{\id,+},\\
    &\psi^-_{(14)(23)}\circ g_{4314}^{\id,+}\circ \psi^-_{(14)(23)}=g_{4314}^{\id,-}\circ g_{1234}^{\id,+}.  
    \end{align*}
    Therefore, it follows that $\Aut(\mathcal{A})\cong D_\infty \rtimes \mathbb{Z}_2$.

    We will see that $\mathcal{B}(t_0)$ is infinite. Consequently, the finiteness of $\mathcal{B}(t_0)$ is not a necessary condition for the finite generation of $\Aut(\mathcal{A})$.
    
    Let 
    \[
\xymatrix{t_0\ar[r]^2&1\ar[r]^1&2\ar[r]^2&3\ar[r]^1&4\ar[r]^2&5\ar[r]^1&6\ar[r]^2&7\ar[r]^1&\cdots}
    \]
    be a subtree of the $4$-regular tree $\mathbb{T}_4$. For each $i\in \mathbb{N}$, we denote by $B_i$ the exchange matrix at $i$, and 
    denote by $w_i$ the sum of the weights in the weighted directed quiver $\Gamma_{B_i}$ associated to $B_i$. A direct computation shows that
     \begin{itemize}
         \item $w_{i+1}>w_i$ for all $i\in \mathbb{N}$. In particular, $[B_i]\neq [B_j]$ whenever $i\neq j\in \mathbb{N}$.
         \item For each $i\in \mathbb{N}$, there is no arrows between vertices $4$ and $1$ or $2$ in $\Gamma_{B_i}$.
     \end{itemize}
     For $i\geq 1$, denote 
     \[t_i = \begin{cases} 
\mu_4(\mu_2\mu_1)^{\frac{i}{2}} \mu_4 (\mu_1\mu_2)^{\frac{i}{2}}, & \text{if } i \in 2\mathbb{Z}; \\
\mu_4\mu_2(\mu_1\mu_2)^{\frac{i-1}{2}} \mu_4 (\mu_2\mu_1)^{\frac{i-1}{2}} \mu_2, & \text{if } i \in 2\mathbb{Z}+1.
\end{cases}\]
It is routine to see that $p(t_0,t_i)\in \mathcal{P}_1(t_0)$ for all $i\geq 1$. It follows that $\mathcal{B}(t_0)$ is an infinite set since $[B_i]\in \mathcal{B}(t_0)$. However, it is worth noting that the automorphisms induced by the paths $p(t_0,t_i)$ all belong to the subgroup $G_0(t_0)$. 

\end{example}

\begin{example}\label{exam:type-X7}
 We end this section by trying to calculate the cluster automorphism group for the cluster algebra $\mathcal{A}$ of type $X_7$. Its weighted direct graph is given in Figure \ref{fig:$X_7$}.

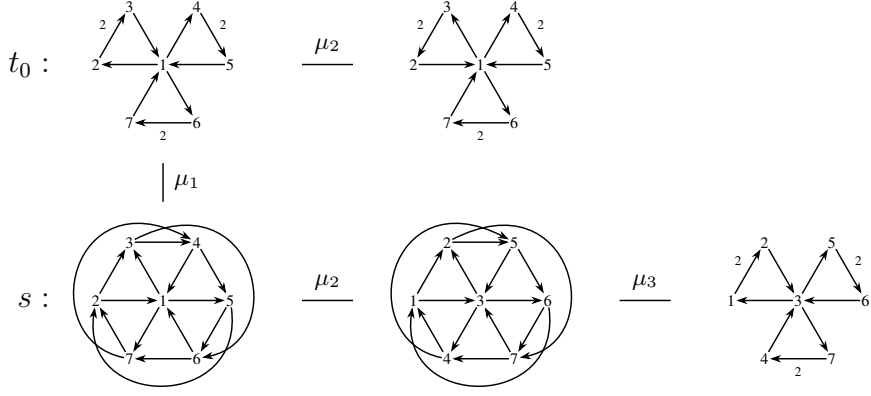
\begin{figure}[htbp]
\centering
\begin{tikzpicture}[
    >=Latex,
    v/.style={inner sep=0.3pt, font=\fontsize{6}{7}\selectfont},
    lab/.style={fill=white, inner sep=0.6pt, font=\fontsize{5}{6}\selectfont},
    edge/.style={-{Stealth[length=1.5mm,width=1.0mm]}, line width=0.55pt},
    aux/.style={line width=0.55pt},
    mulab/.style={font=\scriptsize}
]

% 左侧总标签
\node[left] at (-0.55,0) {$t_0:$};
\node[left] at (-0.55,-3.12) {$s:$};

\begin{scope}[xshift=0.8cm]

%==================================================
% 图 1：第一行左图（t_0 的第一个图）
%==================================================
\begin{scope}[shift={(0,0)}, scale=0.667]
    \node[v] (a1) at (0,0) {1};
    \node[v] (a2) at (-1.333,0) {2};
    \node[v] (a3) at (-0.667,1.155) {3};
    \node[v] (a4) at (0.667,1.155) {4};
    \node[v] (a5) at (1.333,0) {5};
    \node[v] (a6) at (0.667,-1.155) {6};
    \node[v] (a7) at (-0.667,-1.155) {7};

    \draw[edge] (a2) -- node[lab,midway,above left=2.2pt] {2} (a3);
    \draw[edge] (a3) -- (a1);
    \draw[edge] (a1) -- (a2);

    \draw[edge] (a1) -- (a4);
    \draw[edge] (a4) -- node[lab,midway,above right=2.2pt] {2} (a5);
    \draw[edge] (a5) -- (a1);

    \draw[edge] (a7) -- (a1);
    \draw[edge] (a1) -- (a6);
    \draw[edge] (a6) -- node[lab,midway,below=2.2pt] {2} (a7);
\end{scope}

% 图1 -> 图2 的变换标签
\draw[aux] (1.8375,0) -- (2.5125,0);
\node[mulab,above] at (2.175,0) {$\mu_2$};

%==================================================
% 图 2：第一行右图（t_0 的第二个图）
%==================================================
\begin{scope}[shift={(4.2,0)}, scale=0.667]
    \node[v] (b1) at (0,0) {1};
    \node[v] (b2) at (-1.333,0) {2};
    \node[v] (b3) at (-0.667,1.155) {3};
    \node[v] (b4) at (0.667,1.155) {4};
    \node[v] (b5) at (1.333,0) {5};
    \node[v] (b6) at (0.667,-1.155) {6};
    \node[v] (b7) at (-0.667,-1.155) {7};

    \draw[edge] (b2) -- (b1);
    \draw[edge] (b1) -- (b3);
    \draw[edge] (b3) -- node[lab,midway,above left=2.2pt] {2} (b2);

    \draw[edge] (b1) -- (b4);
    \draw[edge] (b4) -- node[lab,midway,above right=2.2pt] {2} (b5);
    \draw[edge] (b5) -- (b1);

    \draw[edge] (b7) -- (b1);
    \draw[edge] (b1) -- (b6);
    \draw[edge] (b6) -- node[lab,midway,below=2.2pt] {2} (b7);
\end{scope}

% 图1 -> 图3 的变换标签
\draw[aux] (0.0,-1.305) -- (0.0,-1.815);
\node[mulab,right] at (0.0,-1.56) {$\mu_1$};

%==================================================
% 图 3：第二行左图（S 的第一个图）
%==================================================
\begin{scope}[shift={(0,-3.12)}, scale=0.667]
    \node[v] (c1) at (0,0) {1};
    \node[v] (c2) at (-1.333,0) {2};
    \node[v] (c3) at (-0.667,1.155) {3};
    \node[v] (c4) at (0.667,1.155) {4};
    \node[v] (c5) at (1.333,0) {5};
    \node[v] (c6) at (0.667,-1.155) {6};
    \node[v] (c7) at (-0.667,-1.155) {7};

    % 六边形边
    \draw[edge] (c2) -- (c3);
    \draw[edge] (c3) -- (c4);
    \draw[edge] (c4) -- (c5);
    \draw[edge] (c5) -- (c6);
    \draw[edge] (c6) -- (c7);
    \draw[edge] (c7) -- (c2);

    % 与中心点 1 相连的边
    \draw[edge] (c2) -- (c1);
    \draw[edge] (c4) -- (c1);
    \draw[edge] (c6) -- (c1);

    \draw[edge] (c1) -- (c3);
    \draw[edge] (c1) -- (c5);
    \draw[edge] (c1) -- (c7);

    % 外包弧线
    \draw[edge] (c3) .. controls (1.85,2.45) and (2.55,-0.78) .. (c6);
    \draw[edge] (c5) .. controls (1.65,-2.20) and (-1.65,-2.20) .. (c2);
    \draw[edge] (c7) .. controls (-2.55,-0.78) and (-1.85,2.45) .. (c4);
\end{scope}

% 图3 -> 图4 的变换标签
\draw[aux] (1.8375,-3.12) -- (2.5125,-3.12);
\node[mulab,above] at (2.175,-3.12) {$\mu_2$};

%==================================================
% 图 4：第二行中图（S 的第二个图）
%==================================================
\begin{scope}[shift={(4.2,-3.12)}, scale=0.667]
    \node[v] (d3) at (0,0) {3};
    \node[v] (d1) at (-1.333,0) {1};
    \node[v] (d2) at (-0.667,1.155) {2};
    \node[v] (d5) at (0.667,1.155) {5};
    \node[v] (d6) at (1.333,0) {6};
    \node[v] (d7) at (0.667,-1.155) {7};
    \node[v] (d4) at (-0.667,-1.155) {4};

    % 外圈边
    \draw[edge] (d1) -- (d2);
    \draw[edge] (d2) -- (d5);
    \draw[edge] (d5) -- (d6);
    \draw[edge] (d6) -- (d7);
    \draw[edge] (d7) -- (d4);
    \draw[edge] (d4) -- (d1);

    % 与中心点 3 相连的边
    \draw[edge] (d1) -- (d3);
    \draw[edge] (d5) -- (d3);
    \draw[edge] (d7) -- (d3);

    \draw[edge] (d3) -- (d2);
    \draw[edge] (d3) -- (d6);
    \draw[edge] (d3) -- (d4);

    % 外包弧线
    \draw[edge] (d2) .. controls (1.85,2.45) and (2.55,-0.78) .. (d7);
    \draw[edge] (d6) .. controls (1.65,-2.20) and (-1.65,-2.20) .. (d1);
    \draw[edge] (d4) .. controls (-2.55,-0.78) and (-1.85,2.45) .. (d5);
\end{scope}

% 图4 -> 图5 的变换标签
\draw[aux] (6.0375,-3.12) -- (6.7125,-3.12);
\node[mulab,above] at (6.375,-3.12) {$\mu_3$};

%==================================================
% 图 5：第二行右图（S 的第三个图）
%==================================================
\begin{scope}[shift={(8.4,-3.12)}, scale=0.667]
    \node[v] (e3) at (0,0) {3};
    \node[v] (e2) at (-0.667,1.155) {2};
    \node[v] (e1) at (-1.333,0) {1};
    \node[v] (e5) at (0.667,1.155) {5};
    \node[v] (e6) at (1.333,0) {6};
    \node[v] (e7) at (0.667,-1.155) {7};
    \node[v] (e4) at (-0.667,-1.155) {4};

    \draw[edge] (e1) -- node[lab,midway,above left=2.2pt] {2} (e2);
    \draw[edge] (e2) -- (e3);
    \draw[edge] (e3) -- (e1);

    \draw[edge] (e3) -- (e5);
    \draw[edge] (e5) -- node[lab,midway,above right=2.2pt] {2} (e6);
    \draw[edge] (e6) -- (e3);

    \draw[edge] (e3) -- (e7);
    \draw[edge] (e7) -- node[lab,midway,below=2.2pt] {2} (e4);
    \draw[edge] (e4) -- (e3);
\end{scope}

\end{scope}
\end{tikzpicture}
\caption{The mutation of $X_7$}
\label{fig:$X_7$}
\end{figure}

Let $G^+_0(t_0) := \{ \psi^+_{\sigma} \mid \sigma B_{t_0} = B_{t_0} \}$ be a normal subgroup of $G_0(t_0)$. Due to the symmetry of the quiver $X_7$, it is clear that $G^+_0(t_0) \cong S_3$. If we denote $a = \psi^+_{(24)(35)}$ and $b = \psi^+_{(46)(57)}$, then $G^+_0(t_0)$ admits the presentation $\langle a, b \mid a^2 = b^2 = (ab)^3 = \mathbf{id} \rangle$. Furthermore, $G_0(t_0)$ can be decomposed as a direct product:$$  G_0(t_0) = G_0^+(t_0) \times \langle \tau \rangle,$$where $\tau = \psi^-_{(23)(45)(67)}$ is a cluster automorphism of order $2$ that commutes with both $a$ and $b$.

A direct computation shows that $|\mathcal{B}(t_0)|=2$ and $l=1$. Hence $\operatorname{Aut}(\mathcal{A})$ is generated by $G_0(t_0)\cup H_1(t_0)^{<4}$. Another straightforward computation yields $[B_{\mu_i(t_0)}]=[B_{t_0}]$ for $i\neq 1$, and $[B_s]=[B_{\mu_i(s)}]$ for $i\neq 1$, where $s=\mu_1(t_0)$. Furthermore, taking $\mu_2(s)$ as an example, we have \begin{align*}
[B_{\mu_i\mu_2(s)}]=
    \begin{cases}
    [B_{t_0}] &\text{$i=3$};\\
    [B_{s}] &\text{else}. 
    \end{cases}
\end{align*}
It follows that $|H_1(t_0)^{<4}|=12$. These automorphisms are associated with the mutation sequences $\mu_i$ for $i \neq 1$, as well as: 
\[
 \mu_3\mu_2\mu_1, \mu_2\mu_3\mu_1, \mu_5\mu_4\mu_1, \mu_4\mu_5\mu_1,\mu_7\mu_6\mu_1,\mu_6\mu_7\mu_1.
 \]
A direct computation further shows that:
\begin{align*}
    & g_3^{(23),+}=\left(g_2^{(23),+}\right)^{-1},\\
    & g_4^{(45),+}=\psi^+_{(24)(35)}\circ g_2^{(23),+}\circ \psi^+_{(24)(35)}.
\end{align*}
Similarly, $\forall i \in \{2,3,4,5,6,7\}$, $g_i^{\sigma,\varepsilon}$ is generated by $G_0(t_0)\cup \{g_2^{(23),+}\}$. Moreover, 
\begin{align*}
&g^{(123)(4567),+}_{231}=\left(g^{(132)(4567),+}_{321}\right)^{-1}\circ \psi^+_{(46)(57)},\\
&g^{(154)(2763),+}_{541}=\psi^+_{(24)(35)}\circ g^{(132)(4567),+}_{321}\circ \psi^+_{(264)(375)}.
\end{align*}
Analogously, one can show that the remaining cluster automorphisms associated with paths of length $3$ passing through the vertex $s$ are also generated by  $G_0(t_0)\cup \{g^{(132)(4567),+}_{321}\}$.

A more involved calculation yields the following relation:
\[
    g_2^{(23),+}\circ \psi^+_{(46)(57)} = \psi^+_{(46)(57)} \circ g_2^{(23),+}=\left(g^{(132)(4567),+}_{321}\right)^2.
\]
Considering the restriction of $g_2^{(23),+}$ to the rank $2$ case, it follows  that $g_2^{(23),+}$ has infinite order, cf. \cite[Lemma 4.1]{FT26}. Consequently, $g^{(132)(4567),+}_{321}$ is also of infinite order.

For simplicity, denote $f=g^{(132)(4567),+}_{321}$. Then $\operatorname{Aut}(\mathcal{A})$ is generated by $\{a,b,\tau,f\}$. It can be checked that
\begin{align*}
b \circ f &=f \circ b,\\
\tau \circ f &=f^{-1}\circ \tau.
\end{align*}

To determine the defining relations of $\operatorname{Aut}(\mathcal{A})$, it remains to characterize the relations between $a$ and $f$. A tedious but straightforward computation shows that
\begin{align*}
& \left( af \right)^5=\mathbf{id}, \\
& \left( afaf^{-1} \right)^3=\mathbf{id}, \\
& \left( af^2\right)^3=\left(f^2a\right)^3,\\
& \left( af^{2}af^{-2} \right)^3=\mathbf{id}, \\
& \left( af^{3}af^{-3} \right)^3=\mathbf{id},\\
& \left( af^{4}\right)^2=\left(f^{4}a \right)^2,\\
& \left( af^{6}af^{-6} \right)^3=\mathbf{id}.
\end{align*}
The complexity of these relations indicates that finding a full presentation is no trivial matter. 
Our extensive computations suggest that the following pattern may hold: for every positive integer $k$,
\begin{align*}
    (\left(af^{2k} \right)^k \left(af^{-2k} \right)^k)^3=\mathbf{id} &\quad \text{if}\ k\equiv 1\ \text{or}\ 5 \pmod{6};\\
    \left(af^{2k} \right)^k \left(af^{-2k} \right)^k=\mathbf{id} &\quad \text{otherwise}.
\end{align*}
Consequently, it may be more prudent to first address the question of whether $\operatorname{Aut}(\mathcal{A})$ is finitely presented.

\end{example}

\bibliographystyle{plain}
\bibliography{myref}

@incollection {FG06,
    AUTHOR = {Fock, V. V. and Goncharov, A. B.},
     TITLE = {Cluster {$\mathcal X$}-varieties, amalgamation, and
              {P}oisson-{L}ie groups},
 BOOKTITLE = {Algebraic geometry and number theory},
    SERIES = {Progr. Math.},
    VOLUME = {253},
     PAGES = {27--68},
 PUBLISHER = {Birkh\"auser Boston, Boston, MA},
      YEAR = {2006},
      ISBN = {978-0-8176-4471-0; 0-8176-4471-7},
   MRCLASS = {22E46 (05E15 20G42 53D17)},
  MRNUMBER = {2263192},
       DOI = {10.1007/978-0-8176-4532-8\_2},
       URL = {https://doi.org/10.1007/978-0-8176-4532-8_2},
}

@article {Fr20,
    AUTHOR = {Fraser, Chris},
     TITLE = {Braid group symmetries of {G}rassmannian cluster algebras},
   JOURNAL = {Selecta Math. (N.S.)},
  FJOURNAL = {Selecta Mathematica. New Series},
    VOLUME = {26},
      YEAR = {2020},
    NUMBER = {2},
     PAGES = {Paper No. 17, 51},
      ISSN = {1022-1824,1420-9020},
   MRCLASS = {13F60 (14M15)},
  MRNUMBER = {4066538},
MRREVIEWER = {Fan\ Qin},
       DOI = {10.1007/s00029-020-0542-3},
       URL = {https://doi.org/10.1007/s00029-020-0542-3},
}

@article{CK06,
 author = {Caldero, Philippe and Keller, Bernhard},
 title = {From triangulated categories to cluster algebras. {II}.},
 fjournal = {Annales Scientifiques de l'{\'E}cole Normale Sup{\'e}rieure. Quatri{\`e}me S{\'e}rie},
 journal = {Ann. Sci. {\'E}c. Norm. Sup{\'e}r. (4)},
 issn = {0012-9593},
 volume = {39},
 number = {6},
 pages = {983--1009},
 year = {2006},
 language = {English},
 doi = {10.1016/j.ansens.2006.09.003},
 url = {https://eudml.org/doc/82705},
 zbMATH = {5149415},
 Zbl = {1115.18301}
}

@article{BD15,
	author = {Blanc, Jeremy and Dolgachev, Igor},
	doi = {10.1007/s00031-014-9289-2},
	eissn = {1531-586X},
	issn = {1083-4362},
	journal = {Transformation Groups},
	number = {1},
	orcid-numbers = {Blanc, Jeremy/0000-0003-3795-7423},
	pages = {1-20},
	title = {AUTOMORPHISMS OF CLUSTER ALGEBRAS OF RANK 2},
	unique-id = {WOS:000350699700001},
	volume = {20},
	year = {2015}}

@article {FZ02,
    AUTHOR = {Fomin, Sergey and Zelevinsky, Andrei},
     TITLE = {Cluster algebras. {I}. {F}oundations},
   JOURNAL = {J. Amer. Math. Soc.},
  FJOURNAL = {Journal of the American Mathematical Society},
    VOLUME = {15},
      YEAR = {2002},
    NUMBER = {2},
     PAGES = {497--529},
      ISSN = {0894-0347,1088-6834},
   MRCLASS = {16S99 (14M99 17B99)},
  MRNUMBER = {1887642},
MRREVIEWER = {Eric\ N.\ Sommers},
       DOI = {10.1090/S0894-0347-01-00385-X},
       URL = {https://doi.org/10.1090/S0894-0347-01-00385-X},
}

@article {FZ03,
    AUTHOR = {Fomin, Sergey and Zelevinsky, Andrei},
     TITLE = {Cluster algebras. {II}. {F}inite type classification},
   JOURNAL = {Invent. Math.},
  FJOURNAL = {Inventiones Mathematicae},
    VOLUME = {154},
      YEAR = {2003},
    NUMBER = {1},
     PAGES = {63--121},
      ISSN = {0020-9910,1432-1297},
   MRCLASS = {17B20 (05E15 16S99 52B12)},
  MRNUMBER = {2004457},
MRREVIEWER = {Eric\ N.\ Sommers},
       DOI = {10.1007/s00222-003-0302-y},
       URL = {https://doi.org/10.1007/s00222-003-0302-y},
}

@article {FZ07,
    AUTHOR = {Fomin, Sergey and Zelevinsky, Andrei},
     TITLE = {Cluster algebras. {IV}. {C}oefficients},
   JOURNAL = {Compos. Math.},
  FJOURNAL = {Compositio Mathematica},
    VOLUME = {143},
      YEAR = {2007},
    NUMBER = {1},
     PAGES = {112--164},
      ISSN = {0010-437X,1570-5846},
   MRCLASS = {16S99 (05E15 14M17 22E46)},
  MRNUMBER = {2295199},
MRREVIEWER = {Christof\ Gei\ss},
       DOI = {10.1112/S0010437X06002521},
       URL = {https://doi.org/10.1112/S0010437X06002521},
}

@article {FT26,
    AUTHOR = {Fu, Changjian and Tang, Xunyi},
     TITLE = {On cluster automorphism groups for decomposable exchange matrices},
   JOURNAL = {accepted by Adv in Math(Chinese)},
  YEAR = {2026},
}

@article {FL25,
    AUTHOR = {Fu, Changjian and Liang, Zhanhong},
     TITLE = {Pseudo $\mathbb{N}$-grading on cluster automorphism groups with application to cluster algebras of rank $3$},
    YEAR={2025},
   JOURNAL = {arXiv:2503.22101(2025)},
}

@article{ASS12,
	author = {Assem, Ibrahim and Schiffler, Ralf and Shramchenko, Vasilisa},
	doi = {10.1112/plms/pdr049},
	issn = {0024-6115},
	journal = {Proceeding of the London Mathematical Society},
	number = {6},
	pages = {1271-1302},
	title = {Cluster automorphisms},
	unique-id = {WOS:000305031100006},
	volume = {104},
	year = {2012}}

@article {Gu11,
    AUTHOR = {Gu, Weiwen},
     TITLE = {A decomposition algorithm for the oriented adjacency graph of
              the triangulations of a bordered surface with marked points},
   JOURNAL = {Electron. J. Combin.},
  FJOURNAL = {Electronic Journal of Combinatorics},
    VOLUME = {18},
      YEAR = {2011},
    NUMBER = {1},
     PAGES = {Paper 91, 45},
      ISSN = {1077-8926},
   MRCLASS = {05C10 (05C85)},
  MRNUMBER = {2795772},
       DOI = {10.37236/578},
       URL = {https://doi.org/10.37236/578},
}

@article {BS15,
    AUTHOR = {Bridgeland, Tom and Smith, Ivan},
     TITLE = {Quadratic differentials as stability conditions},
   JOURNAL = {Publ. Math. Inst. Hautes \'Etudes Sci.},
  FJOURNAL = {Publications Math\'ematiques. Institut de Hautes \'Etudes
              Scientifiques},
    VOLUME = {121},
      YEAR = {2015},
     PAGES = {155--278},
      ISSN = {0073-8301,1618-1913},
   MRCLASS = {14D20 (14N35 18E30 57M50 81T20)},
  MRNUMBER = {3349833},
MRREVIEWER = {Brent\ Pym},
       DOI = {10.1007/s10240-014-0066-5},
       URL = {https://doi.org/10.1007/s10240-014-0066-5},
}

@article {CZ16a,
    AUTHOR = {Chang, Wen and Zhu, Bin},
     TITLE = {Cluster automorphism groups of cluster algebras of finite
              type},
   JOURNAL = {J. Algebra},
  FJOURNAL = {Journal of Algebra},
    VOLUME = {447},
      YEAR = {2016},
     PAGES = {490--515},
      ISSN = {0021-8693,1090-266X},
   MRCLASS = {13F60},
  MRNUMBER = {3427647},
MRREVIEWER = {Ralf\ Schiffler},
       DOI = {10.1016/j.jalgebra.2015.09.045},
       URL = {https://doi.org/10.1016/j.jalgebra.2015.09.045},
}

@article {BM16,
    AUTHOR = {Bazier-Matte, V\'eronique},
     TITLE = {Unistructurality of cluster algebras of type {$\tilde{\Bbb
              A}$}},
   JOURNAL = {J. Algebra},
  FJOURNAL = {Journal of Algebra},
    VOLUME = {464},
      YEAR = {2016},
     PAGES = {297--315},
      ISSN = {0021-8693,1090-266X},
   MRCLASS = {13F60},
  MRNUMBER = {3533433},
MRREVIEWER = {Jianrong\ Li},
       DOI = {10.1016/j.jalgebra.2016.06.026},
       URL = {https://doi.org/10.1016/j.jalgebra.2016.06.026},
}

@article {BMP20,
    AUTHOR = {Bazier-Matte, V\'eronique and Plamondon, Pierre-Guy},
     TITLE = {Unistructurality of cluster algebras from unpunctured
              surfaces},
   JOURNAL = {Proc. Amer. Math. Soc.},
  FJOURNAL = {Proceedings of the American Mathematical Society},
    VOLUME = {148},
      YEAR = {2020},
    NUMBER = {6},
     PAGES = {2397--2409},
      ISSN = {0002-9939,1088-6826},
   MRCLASS = {13F60},
  MRNUMBER = {4080883},
MRREVIEWER = {Christof\ Gei\ss},
       DOI = {10.1090/proc/14932},
       URL = {https://doi.org/10.1090/proc/14932},
}

@article {CL2020,
    AUTHOR = {Cao, Peigen and Li, Fang},
     TITLE = {Unistructurality of cluster algebras},
   JOURNAL = {Compos. Math.},
  FJOURNAL = {Compositio Mathematica},
    VOLUME = {156},
      YEAR = {2020},
    NUMBER = {5},
     PAGES = {946--958},
      ISSN = {0010-437X,1570-5846},
   MRCLASS = {13F60},
  MRNUMBER = {4089373},
MRREVIEWER = {L\'ea\ Bittmann},
       DOI = {10.1112/s0010437x20007113},
       URL = {https://doi.org/10.1112/s0010437x20007113},
}

@article {DL23,
    AUTHOR = {Dong, Jinlei and Li, Fang},
     TITLE = {Presentations of mapping class groups and an application to cluster algebras from surfaces},
    JOURNAL = {J. Algebra},
    VOLUME = {663},
      YEAR = {2025},
     PAGES = {882--912},
}

@article {FST12a,
    AUTHOR = {Felikson, Anna and Shapiro, Michael and Tumarkin, Pavel},
     TITLE = {Skew-symmetric cluster algebras of finite mutation type},
   JOURNAL = {J. Eur. Math. Soc. (JEMS)},
  FJOURNAL = {Journal of the European Mathematical Society (JEMS)},
    VOLUME = {14},
      YEAR = {2012},
    NUMBER = {4},
     PAGES = {1135--1180},
      ISSN = {1435-9855,1435-9863},
   MRCLASS = {13F60 (05C30 05E15)},
  MRNUMBER = {2928847},
MRREVIEWER = {Xueqing\ Chen},
       DOI = {10.4171/JEMS/329},
       URL = {https://doi.org/10.4171/JEMS/329},
}

@article {FST12b,
    AUTHOR = {Felikson, Anna and Shapiro, Michael and Tumarkin, Pavel},
     TITLE = {Cluster algebras of finite mutation type via unfoldings},
   JOURNAL = {Int. Math. Res. Not. IMRN},
  FJOURNAL = {International Mathematics Research Notices. IMRN},
      YEAR = {2012},
    NUMBER = {8},
     PAGES = {1768--1804},
      ISSN = {1073-7928,1687-0247},
   MRCLASS = {13F60},
  MRNUMBER = {2920830},
MRREVIEWER = {Yu\ Zhou},
       DOI = {10.1093/imrn/rnr072},
       URL = {https://doi.org/10.1093/imrn/rnr072},
}

@article {I20,
    AUTHOR = {Ishibashi, Tsukasa},
     TITLE = {Presentations of cluster modular groups and generation by
              cluster {D}ehn twists},
   JOURNAL = {SIGMA Symmetry Integrability Geom. Methods Appl.},
  FJOURNAL = {SIGMA. Symmetry, Integrability and Geometry. Methods and
              Applications},
    VOLUME = {16},
      YEAR = {2020},
     PAGES = {Paper No. 025, 22},
      ISSN = {1815-0659},
   MRCLASS = {13F60 (05E14 30F60)},
  MRNUMBER = {4082289},
MRREVIEWER = {Ashish\ K.\ Srivastava},
       DOI = {10.3842/SIGMA.2020.025},
       URL = {https://doi.org/10.3842/SIGMA.2020.025},
}
\end{document}